\documentclass[10pt]{article}
\usepackage{amssymb,amsthm,amsmath}
\usepackage{color,graphicx}
\usepackage{authblk}
\usepackage[dvipsnames]{xcolor}
\usepackage{tikz,cancel}
\usepackage{mathrsfs}
\usetikzlibrary{patterns}
\usepackage{mathtools}
\usepackage{caption}
\usepackage{lineno}
\usepackage{hyperref}
 \hypersetup{linkcolor=blue, colorlinks=true,citecolor = red}
\usepackage[capitalise]{cleveref}
\usepackage{upref}
\usepackage{changes}
\usepackage{mathrsfs}

\newtheorem{corollary}{Corollary}[section]
\newtheorem{definition}[corollary]{Definition}

\newtheorem{proposition}[corollary]{Proposition}

\newtheorem{theorem}[corollary]{Theorem}
\numberwithin{equation}{section}

\renewcommand{\leq}{\leqslant}

\renewcommand{\geq}{\geqslant}

\makeatletter
\def\section{\@startsection {section}{1}{\z@}{-3.5ex plus -1ex minus
    -.2ex}{2.3ex plus .2ex}{\large\bf}}

\makeatother

\begin{document}

%%%%%%%%%+%%%%%%%%%+%%%%%%%%%+%%%%%%%%%+%%%%%%%%%+%%%%%%%%%+%%%%%%%%%+
\title{\bf  Input-to-state type Stability for Simplified Fluid-Particle Interaction System }
\author[1]{Zhuo Xu}
\affil[1]{Institut de Math\'{e}matiques de Bordeaux UMR 5251, Universit\'{e} de Bordeaux/Bordeaux INP/CNRS, 351 Cours de la Lib\'{e}ration, 33 405 TALENCE, France}
%zhuo.xu@math.u-bordeaux.fr}
\date{}
%%%%%%%%%%%%%%%%%%%%%%%%%%%%%%%%%%%%%%%%%%%%%%%%%%%

\maketitle

\begin{abstract}
\vskip 0.2in
\par In this paper, we study the well-posedness and the input-to-state type stability of a one-dimensional fluid-particle interaction system. A distinctive feature, not yet considered in the ISS literature,  is that our system involves a free boundary. More precisely, the fluid is described by the viscous Burgers equation and the motion of particle obeys Newton’s second law. The point mass is subject to both a feedback control and an open-loop control. We first establish the well-posedness of the system for any open-loop input in \( L^2(0,\infty) \). Then, assuming the input also belongs to \( L^1(0,\infty) \), we prove that the particle's position remains uniformly bounded and that the system is input-to-state type stable. The proof is based on the construction of a Lyapunov functional derived from a special test function.
%Finally, we establish a local input-to-state stability result under stronger assumptions on the feedback control and %initial data.

\end{abstract}



\tableofcontents
 \section{Introduction}

In recent years, the concept of input-to-state stability (ISS) has attracted a lot of attention in both mathematics and engineering. It was first introduced by Sontag \cite{sontag1989smooth}, providing a framework to analyze how external inputs influence the internal state of a system, particularly those described by ordinary differential equations. A major breakthrough of the ISS theory is the equivalence between ISS and the existence of a smooth ISS Lyapunov function by Sontag and Wang \cite{sontag1995characterizations}.

The ISS framework has been extended to infinite-dimensional systems, as can be seen in recent works developing appropriate functional analytic tools and Lyapunov-based characterizations; see Jacob, Mironchenko and Schwenninger \cite{jacob2022input}, Jacob, Nabiullin, Partington and Schwenninger \cite{jacob2018infinite}, Jacob, Nabiullin,  Partington, and Schwenninger \cite{jacob2016input}.  These advances have enabled the study of abstract evolution systems, boundary control systems, and distributed parameter systems in a unified way. More recently, ISS theory has been extended to nonlinear systems, Jacob, Tucsnak, Hosfeld and Schwenninger \cite{tucsnak2023input}, they show the input to state stable estimate to bilinear feedback system under small initial condition and they apply the result to some nonlinear systems.

The study of ISS property for nonlinear partial differential equations also attract a lot of attention. Karafyllis and Krstic \cite{karafyllis2016input}, they focused on parabolic PDEs with boundary disturbances and established the iss property for the nonlinear system. Later, they extensions to broader classes of 1-D parabolic PDEs presented in Karafyllis and Krstic \cite{karafyllis2016iss, karafyllis2017iss}. Further developments include the use of monotonicity methods to address nonlinearities and boundary inputs Mironchenko, Karafyllis and Krstic \cite{mironchenko2019monotonicity}, and the derivation of ISS estimates in spatial sup-norms for nonlinear parabolic systems Karafyllis and Krstic \cite{karafyllis2021iss}.

In this paper, we study the ISS type properties of an one dimension simplified fluid–particle interaction nonlinear system. This work continues the study by Tucsnak and Xu \cite{xu5114347global}, in which the authors addressed the global well-posedness and exponential stabilization of the closed-loop system. Our aim in this paper is on the open-loop control system, which means we add another open loop control for the feedback system. We aim to analyze the robustness of the system with respect to such external perturbations.

The fluid–particle interaction model under consideration was first introduced in V{\'a}zquez and Zuazua \cite{vazquez2003large, zuazua2006}, where the existence, uniqueness, and large-time behavior of solutions were studied in the case where the fluid occupies the entire space. Later, Doubova and Fernández-Cara \cite{doubova2005some} investigated the boundary controllability of this system, proving local null controllability for the nonlinear system when controls are applied at both ends of the spatial domain. The single boundary control problem was studied by Liu, Takahashi and Tucsnak \cite{liu2013single}, who employed spectral analysis and introduced the so-called \emph{source point method}.

In 2015, C{\^i}ndea, Micu, Roven{\c{t}}a and Tucsnak \cite{cindea2015particle} analyzed the model where the control acts directly on the particle. Their study focused on a closed-loop system with initial data close to equilibrium, and they established local null controllability using a combination of spectral and fixed-point techniques.

More precisely, we consider the following system:
\begin{align}
\label{conversition_of_mom_of_solid_without_control}	&\dot{v}(t,y) + v(t,y) v_y(t,y)- v_{yy}(t,y)  = 0 & t &\in (0,T), \quad y \in (-1,1), \quad y \neq h(t), \\
	\label{boundary_condition_of_fluid_without_control}	&v(t,-1) = v(t,1) = 0 & t &\in (0,T), \\
	\label{campability_condition_without_control}&\dot{h}(t) = v(t,h(t)) & t &\in (0,T), \\
	\label{newton_second_law_without_control}&\ddot{h}(t) = [v_y](t,h(t))+\mathcal{K}(h_1-h(t))
	+u(t)  & t &\in (0,T), \\
	\label{inlitial_date_of_veloctiy_without_control}&v(0,y) = v_0(y) & y &\in (-1,1),\\
	\label{inlitial_date_of_solid_without_control}&h(0)=h_0,\qquad \qquad \dot{h}(0)=g_0.
\end{align}
where $\mathcal{K}\geq 0$ is a  given constant and
\begin{equation}\label{definition_of_jump}
\left[v_y\right](t,h(t))=v_y\left(t,h(t)^+\right)-v_y\left(t,h(t)^-\right)\qquad \qquad (t>0).
\end{equation}

Here, $v = v(t, y)$ represents the velocity field of a viscous fluid, governed by the Burgers equation. The motion of a particle immersed in this fluid is described by Newton’s second law as given in \eqref{newton_second_law_without_control}. The position of the particle is denoted by $h(t)$ and its velocity by $\dot{h}(t)$, both depending only on time. The external input force $u(t)$ also depends solely on time and $h_1\in (-1,1)$ is our target point of particle.

 System \eqref{conversition_of_mom_of_solid_without_control}–\eqref{inlitial_date_of_solid_without_control} constitutes a free boundary value problem because the particle's position $h(t)$, which influences the domain of the fluid, is an unknown of the system. This fact leads us to adopt the notion of a finite energy solution, as defined in \cite[Definition 1.1]{cindea2015particle}:

\begin{definition}\label{definition_of_finite_energy}
	Given $T>0$, $v_0\in L^2(-1,1)$, $g_0\in \mathbb{R}$, $h_0,\,h_1\in (-1,1)$, $\mathcal{K}\geq 0$ and $u\in L^2(0,T)$, we say that
	\begin{equation}
		\left[\begin{gathered}
			v\\
			g\\
			h
		\end{gathered}\right]\in \left[ C\left([0,T];\;L^2(-1,1)\right)\cap  L^2\left((0,T);\;H^1_0(-1,1)\right)\right]\times H^1(0,T)\times L^2(0,T),
	\end{equation}
	is a finite energy solution of system \eqref{conversition_of_mom_of_solid_without_control}–\eqref{inlitial_date_of_solid_without_control} on $[0,T]$ if $h(0)=h_0\in (-1,1)$, $\dot{h}(t)=g(t)=v\left(t,h(t)\right)$, and $h(t)\in (-1,1)$ for almost every $t\in [0,T)$, and if for all $t \in [0,T)$ the following equation holds:
    \begin{align}
		&\int_{-1}^{1}v(t,y)\psi(t,y)\,{\rm d}y-\int_{-1}^{1}v_0(y)\psi(0,y)\,{\rm d}y+g(t)l(t)-g_0l(0)-\int_{0}^{t}g(\sigma)\dot{l}(\sigma)\,{\rm d}\sigma\notag\\
		&-\int_{0}^{t}\int_{-1}^{1}v(\sigma,y)\dot{\psi} (\sigma, y)\,{\rm d}y\,{\rm d}\sigma
		+\int_{0}^{t}\int_{-1}^{1} v_y(\sigma, y)\psi_y(\sigma,y)\,{\rm d}y\,{\rm d}\sigma\notag\\
		&-\frac12\int_{0}^{t}\int_{-1}^{1}v^2(\sigma ,y)\psi_y(\sigma,y)\,{\rm d}y\,{\rm d}\sigma
		=\int_{0}^{t}u(\sigma)l(\sigma)\,{\rm d}\sigma+\int_{0}^{t}\mathcal{K}\left(h_1-h(\sigma)\right)l(\sigma)\,{\rm d}\sigma,
	\end{align}
	for every test function
	\begin{align}
		&\left[\begin{gathered}
			\psi\\
			l
		\end{gathered}\right]\in \left\{ H^1\left((0,T);\;L^2(-1,1)\right)\cap  L^2\left((0,T);\;H^1_0(-1,1)\right)\right\}\times H^1(0,T),\\
		&l(t)=\psi(t,l(t)).
	\end{align}
    A triple $\begin{bmatrix}
			v\\
			g\\
			h
		\end{bmatrix}$ which is a finite energy solution on $[0,T]$ for every $T>0$ is called a global finite energy solution.
\end{definition}

Note that the test functions depend explicitly on the solution, through its component $h(t)$. Throughout the paper, for $m\in \mathbb{N}$, we denote by $H^m$ and $H^m_0$ the usual Sobolev spaces, and by $H^{-m}$ the dual of $H^m_0$ with respect to the pivot space $L^2$.

We now present the main results of the paper. The first theorem ensures the global existence and uniqueness result of system \eqref{conversition_of_mom_of_solid_without_control}-\eqref{inlitial_date_of_solid_without_control}:

\begin{theorem}\label{global_exists}
Let $v_0\in L^2(-1,1)$, $g_0\in\mathbb{R}$, $h_0,\, h_1\in(-1,1)$, $\mathcal{K}\geq 0$ and $u\in L^2(0, \infty)$, then
system \eqref{conversition_of_mom_of_solid_without_control}–\eqref{inlitial_date_of_solid_without_control} admits a unique global finite energy solution as defined above. Moreover, for every $t\geq 0$ the following energy equality holds:
	\begin{align}\label{energy_equality}
		&\int_{-1}^{1}v^2(t,y)\,{\rm d}y-\int_{-1}^{1}v^2_0(y)\,{\rm d}y+g^2(t)-g_0^2+\mathcal{K}\left(h_1-h(t)\right)^2-\mathcal{K}\left(h_1-h_0\right)^2\notag\\
		&+2\int_{0}^{t}\int_{-1}^{1}v_y^2(\sigma ,y)\,{\rm d}y\,{\rm d}\sigma
		=\int_{0}^{t}u(\sigma)g(\sigma)\,{\rm d}\sigma.
	\end{align}
\end{theorem}

The second result implies the boundedness of the particle trajectory under additional assumptions on the input $u$.

\begin{theorem}\label{unforimboundtheorem}
With the same initial conditions as in Theorem~\ref{global_exists}, if the open control $u\in L^2\cap L^1(0, \infty)$, then system \eqref{conversition_of_mom_of_solid_without_control}–\eqref{inlitial_date_of_solid_without_control} admits a unique global finite energy solution. Moreover, there exists $\alpha > 0$, depending on $\|v_0\|_{L^2(-1,1)}$, $g_0$, $h_0$, $h_1$, $\mathcal{K}$, $\|u\|_{L^2(0,\infty)}$, and $\|u\|_{L^1(0,\infty)}$, such that
\begin{align}\label{estimateh1}
	-1+\alpha \leq h(t) \leq 1-\alpha.
\end{align}
\end{theorem}

Before stating the main result of the ISS type, we briefly mention that a formal discussion of the ISS will be provided in Section~\ref{background_of_input}. We now present the main stability theorem of system \eqref{conversition_of_mom_of_solid_without_control}-\eqref{inlitial_date_of_solid_without_control}:

\begin{theorem}\label{theoreminputtostate}
With the notation and assumptions of Theorem \ref{unforimboundtheorem}, system \eqref{conversition_of_mom_of_solid_without_control}–\eqref{inlitial_date_of_solid_without_control} exhibits the following behavior:
\begin{itemize}
	\item If $\mathcal{K}=0$, then there exists a constant $\eta_1 > 0$ such that
	\begin{align}\label{exp_velocity}
		\Vert v(t,\cdot)\Vert^2_{L^2(-1,1)}+g^2(t)
		\leq 16\exp\left(-\frac14 t\right)\left(	\Vert v_0\Vert^2_{L^2(-1,1)}+g^2_0\right)+16 \Vert u\Vert_{L^2(0,\infty)}\qquad (t>0).
	\end{align}
	
	\item If $\mathcal{K}>0$, then there exist constants $\eta > 0$, which are both depending on $\Vert v_0 \Vert_{L^2(-1,1)}$, $g_0$, $h_0$, $h_1$, $\Vert u\Vert_{L^2(0,\infty)}$, $\Vert u\Vert_{L^1(0,\infty)}$, and $\mathcal{K}$, such that
	\begin{align}\label{input_states_stability1}
	&\Vert v(t,\cdot)\Vert^2_{L^2(-1,1)}+g^2(t)+\mathcal{K}(h(t)-h_1)^2\notag\\
	\leq& 16\exp\left(-\eta t\right)\left(	\Vert v_0\Vert^2_{L^2(-1,1)}+g^2_0+\mathcal{K}(h_0-h_1)^2\right)+ \frac32\Vert u\Vert_{L^2(0,\infty)}\qquad (t>0).
	\end{align}
\end{itemize}
\end{theorem}

The above results show that for all $\mathcal{K}\geq 0$ and $u \in L^2(0,\infty)$, system \eqref{conversition_of_mom_of_solid_without_control}-\eqref{inlitial_date_of_solid_without_control} always admits a unique global finite energy solution. In the absence of feedback control (i.e., $\mathcal{K}=0$), we cannot determine the long-time behavior of $h(t)$ purely from the initial data. In contrast, when $\mathcal{K} > 0$, the system satisfies an input-to-state type stability estimate. The reason for which this is not a standard ISS estimate is that the constant $\eta$ in \eqref{input_states_stability1} potentially depends on some norm of $u$. This limitation comes from the fact that the distance from the moving particle to the boundary, although (as shown below) positive for every $t\geqslant 0$ could go to zero when $\|u\|_{L^2(0,\infty)}\to \infty$ or $\|u\|_{L^1(0,\infty)}\to \infty$.

To prove these results, we adapt the methods from \cite[Section 2]{xu5114347global} to establish Theorem~\ref{global_exists}, apply a comparison argument in ODEs to prove Theorem \ref{unforimboundtheorem}. Finally we use Lyapunov-based techniques inspired by \cite[Section 3]{xu5114347global} to obtain the ISS type estimates \eqref{input_states_stability1} in Theorem \ref{theoreminputtostate}.

%%%%%%%%%%%%%%%%%%%%%%%%%%%%%%%%%%%%%%%%%%%%%%%%%%%%%%%%%%%%%%%%%%%%%%%%%%%%%%%%%%%%%%%%%%%%%%%%%%%%%%%%%%%%%%%%%%%%%%%%%%%%%%%%%%%%%%%%%%%%%

%%%%%%%%%%%%%%%%%%%%%%%%%%%%%%%%%%%%%%%%%%%%%%%%%%%%%%%%%%%%%%%%%%%%%%%%%%%%%%%%%%%%%%%%%%%%%%%%%%%%%%%%%%%%%%%%%%%%%%%%%%
\section{Background on Input-to-State Stability}\label{background_of_input}

In this section, we introduce some basic concepts related to input-to-state stability (ISS). We begin with a general definition of a control system. The presentation closely follows the framework in Mironchenko and Prieur~\cite{inputoverview}.

\begin{definition}\label{definitionofsystem}
A control system is defined by the triple $\Sigma = (X, \mathcal{U}, \phi)$, where:
\begin{itemize}
    \item[(i)] $X$ is a normed vector space, called the \emph{state space}, equipped with the norm $\|\cdot\|_X$.
    
    \item[(ii)] $\mathcal{U}$ is a normed vector space of inputs, where each $u \in \mathcal{U}$ is a function $u: \mathbb{R}_+ \to U$ and $U$ is a normed vector space of input values. The space $\mathcal{U}$ satisfies the following two properties:
    \begin{itemize}
        \item[i] For all $u \in \mathcal{U}$ and $\tau \geq 0$, the time-shifted function $u(\cdot + \tau)$ also belongs to $\mathcal{U}$, and $\|u\|_{\mathcal{U}} \geq \|u(\cdot + \tau)\|_{\mathcal{U}}$.
        
        \item[ii] For all $u_1, u_2 \in \mathcal{U}$ and $t > 0$, the concatenation
        \[
        (u_1 \diamondsuit_t u_2)(\tau) :=
        \begin{cases}
        u_1(\tau), & \text{if } \tau \in [0, t], \\
        u_2(\tau - t), & \text{otherwise},
        \end{cases}
        \]
        also belongs to $\mathcal{U}$.
    \end{itemize}
    
    \item[(iii)] $\phi: D_\phi \to X$ is a transition map, where $D_\phi \subset \mathbb{R}_+ \times X \times \mathcal{U}$. For every initial state $x \in X$ and input $u \in \mathcal{U}$, there exists a maximal existence time $t_m = t_m(x,u) \in (0, +\infty]$ such that
    \[
    D_\phi \cap (\mathbb{R}_+ \times \{(x,u)\}) = [0, t_m) \times \{(x,u)\} \subset D_\phi.
    \]
    The interval $[0, t_m)$ is called the \emph{maximal domain of definition} of the map $t \mapsto \phi(t, x, u)$.
\end{itemize}

The triple $\Sigma$ is called a \emph{(control) system} if the following conditions hold:
\begin{itemize}
    \item[$(\Sigma 1)$] For all $(x, u) \in X \times \mathcal{U}$, we have $\phi(0, x, u) = x$.
    
    \item[$(\Sigma 2)$] If $u$ and $\tilde{u}$ agree on $[0, t]$, then $\phi(t, x, u) = \phi(t, x, \tilde{u})$.
    
    \item[$(\Sigma 3)$] For each fixed $(x, u) \in X \times \mathcal{U}$, the map $t \mapsto \phi(t, x, u)$ is continuous on its domain.
    
    \item[$(\Sigma 4)$] For all $x \in X$, $u \in \mathcal{U}$, and $t, h \geq 0$ such that $[0, t+h] \times \{(x, u)\} \subset D_\phi$, it holds that
    \[
    \phi(h, \phi(t, x, u), u(t + \cdot)) = \phi(t + h, x, u).
    \]
\end{itemize}
\end{definition}

This abstract definition covers many important types of systems, such as ordinary differential equations, partial differential equations, time-delay systems, and even more complex systems like interconnected or hybrid systems. A similar definition of nonlinear systems can be found in Section~7 of Tucsnak and Weiss~\cite{tucsnak2014well}. They explain that a nonlinear system can be regarded as a linear system with a nonlinear feedback term.

Next, we introduce a basic well-posedness condition.

\begin{definition}
A control system (as defined above) is said to be \emph{forward complete} if
\[
D_\phi = \mathbb{R}_+ \times X \times \mathcal{U},
\]
i.e., for every $(x, u) \in X \times \mathcal{U}$ and for all $t \geq 0$, the value $\phi(t, x, u) \in X$ is well-defined.
\end{definition}

\vspace{1em}

To define input-to-state stability, we use several standard classes of comparison functions:
\begin{align*}
\mathcal{P} &:= \left\{ \gamma \in C(\mathbb{R}_+) : \gamma(0) = 0, \ \gamma(r) > 0 \text{ for } r > 0 \right\}, \\
\mathcal{K} &:= \left\{ \gamma \in \mathcal{P} : \gamma \text{ is strictly increasing} \right\}, \\
\mathcal{K}_\infty &:= \left\{ \gamma \in \mathcal{K} : \gamma \text{ is unbounded} \right\}, \\
\mathcal{L} &:= \left\{ \gamma \in C(\mathbb{R}_+) : \gamma \text{ is strictly decreasing and } \lim_{t \to \infty} \gamma(t) = 0 \right\}, \\
\mathcal{K}\mathcal{L} &:= \left\{ \beta \in C(\mathbb{R}_+ \times \mathbb{R}_+, \mathbb{R}_+) : 
\begin{array}{l}
\beta(\cdot, t) \in \mathcal{K} \text{ for all } t \geq 0, \\
\beta(r, \cdot) \in \mathcal{L} \text{ for all } r > 0
\end{array}
\right\}.
\end{align*}

Functions in class $\mathcal{P}$ are also called \emph{positive definite functions}.

\vspace{1em}
Now we introduce the main concept of this section: input-to-state stability.

\begin{definition}
System $\Sigma = (X, \mathcal{U}, \phi)$ is said to be \emph{(uniformly) input-to-state stable (ISS)} if there exist functions $\beta \in \mathcal{K}\mathcal{L}$ and $\gamma \in \mathcal{K}_\infty$ such that for all $x \in X$, $u \in \mathcal{U}$, and $t \geq 0$, we have
\begin{equation}
\|\phi(t, x, u)\|_X \leq \beta(\|x\|_X, t) + \gamma(\|u\|_{\mathcal{U}}).
\tag{1.6}
\end{equation}
\end{definition}

This inequality describes how the state of the system behaves over time. Specifically, the first term $\beta(\|x\|_X, t)$ shows that the effect of the initial condition decreases as time increases. The second term $\gamma(\|u\|_{\mathcal{U}})$ shows how the input signal influences the state in the long run.

\section{Proof of Theorem \ref{global_exists}}\label{section_global}
In this section, we establish the global existence and uniqueness of finite energy solutions as defined in \eqref{definition_of_finite_energy}. To begin, we introduce the following theorem, which concerns the local (in time) well-posedness of the system \eqref{conversition_of_mom_of_solid_without_control}–\eqref{inlitial_date_of_solid_without_control}. The theorem is stated as follows:

\begin{theorem}\label{local_existstheorem}
	For every $0<\varepsilon\leq \frac14$, $\kappa>0$ and $h_1\in (-1,1)$, we denote by $\mathcal{B}_{\varepsilon,\kappa}$ the set $\begin{bmatrix}
			v_0\\
			g_0\\
			h_0\\
			u
		\end{bmatrix}\in L^2(-1,1)\times\mathbb{R}\times \mathbb{R}\times L^2(0,\infty)$ satisfying
\begin{align}
	&\left\Vert v_0\right\Vert^2_{L^2(-1,1)}+g^2_0+\Vert u\Vert_{L^2(0,\infty)}+\left(h_1-h(t)\right)^2
    \leq \kappa,\\
	&-1+4\varepsilon \leq h_0\leq 1-4\varepsilon.
\end{align}
Then there exist $T>0$, $T$ only denpends on $\varepsilon$ and $\kappa$, such that for every $\begin{bmatrix}
		v_0\\
		g_0\\
		h_0\\
		u
	\end{bmatrix}\in \mathcal{B}_{\varepsilon,\kappa}$, system \eqref{conversition_of_mom_of_solid_without_control}-\eqref{inlitial_date_of_solid_without_control} admits a unique finite energy solution defined by Definition \ref{definition_of_finite_energy} in the time interval $[0,T]$. Moreover, for every $t\in [0,T]$ we have the following energy estimate:
\begin{align}\label{energy_equality1}
		&\int_{-1}^{1}v^2(t,y)\,{\rm d}y+g^2(t)+\mathcal{K}\left(h_1-h(t)\right)^2+\int_{0}^{t}\int_{-1}^{1}v_y^2(\sigma ,y)\,{\rm d}y{\rm d}\sigma\notag\\
		\leq& \int_{0}^{t}u^2(\sigma)\,{\rm d}\sigma+\int_{-1}^{1}v^2(0,y)\,{\rm d}y+g^2(0)+\mathcal{K}\left(h_1-h_0\right)^2,
	\end{align}
		\end{theorem}
 Here, we do not provide the proof of the above theorem, as it can be established step by step following \cite[Section 2-4]{cindea2015particle}.
 
 As a consequence of the above results, we have:
 
\begin{corollary}\label{corollary_contraction}
	Assume that for every $\tau>0$, a finite energy solution in Definition \ref{definition_of_finite_energy} of system \eqref{conversition_of_mom_of_solid_without_control}-\eqref{inlitial_date_of_solid_without_control} defined on $[0,\tau)$
	satisfies
	\begin{equation}
		\sup_{t \in [0,\tau)}\vert h(t)\vert <1.
	\end{equation}
	Then the considered finite energy solution is global.
\end{corollary}
Our main result in this section can be stated as follows.
\begin{proposition}
    \label{theorem_global}
	Let $v_0\in L^2(-1,1)$, $g_0\in \mathbb{R}$, $h_0,\,h_1\in (-1,1)$ and $u\in L^2(0,\infty)$. Moreover, let $\begin{bmatrix}
	    v\\
        g\\
        h
	\end{bmatrix}$ be the finite energy solution defined by Definition \ref{definition_of_finite_energy} on $(0, \tau )$. Then for every $t\in (0,\tau)$, the position of praticle $h(t)$ statisfies 
	\begin{align}\label{estimateh}
		-1+c_1(t)
		\leq h(t)
		\leq1-c_2(t) \qquad\qquad (t\in (0,\tau ))
	\end{align}
	where
	\begin{align}
		c_1(t)=&\frac{2}{1+\max\{2,\frac{({1+h_0})}{1-h_0}\}\exp\left(C+t^{\frac12}\left(\int_{0}^{t}u^2(s){\rm d}s\right)^{\frac12}\right)},\\
		c_2(t)=&\frac{2}{1+\max\{2,\frac{({1-h_0})}{1+h_0}\}\exp\left(\left(C+t^{\frac12}\left(\int_{0}^{t}u^2(s){\rm d}s\right)^{\frac12}\right)\right)},\notag\\
        C=&10\left(\Vert v_0\Vert_{L^2(-1,1)}^2+g^2_0+\mathcal{K}(h_1-h_0)^2
	\right.\notag\\
&\left.	+ \Vert u\Vert^2_{L^2(0,\infty)}+\sqrt{\Vert v_0\Vert_{L^2(-1,1)}^2+g^2_0+\mathcal{K}(h_1-h_0)^2+ \Vert u\Vert^2_{L^2(0,\infty)}}\right).
	\end{align}
\end{proposition}
\begin{proof}
	Following the approach proposed in \cite[Section II]{Helsa2004PhD75H}, we begin by choosing the test function $\begin{bmatrix} \psi\\ l\end{bmatrix} =\begin{bmatrix}
		\varphi\\
		1
	\end{bmatrix}$ in the Definition \ref{definition_of_finite_energy},
	where $\varphi$ is defined by:
	\begin{equation}\label{testing_function}
		\varphi(t,y)=\begin{cases}
			\frac{1}{1+h(t)}(y+1) & \qquad\qquad(y\in (-1,h(t))),\\
			\frac{1}{1-h(t)}(1-y)                       & \qquad\qquad(y\in(h(t),1)).	
		\end{cases}
	\end{equation}
	It follows that
	\begin{align}\label{intergrate_form}
		&\int_{-1}^{1}v(t,y)\varphi(t,y)\,{\rm d}y
		-
		\int_{-1}^{1}v_0(y)\varphi(0,y)\,{\rm d}y
		-
		\int_{0}^{t}\int_{-1}^{1}v(\sigma, y)\frac{\partial \varphi}{\partial \sigma}(\sigma, y)\,{\rm d}y\,{\rm d}\sigma\notag\\
		&+\int_{0}^{t} \int_{-1}^{1} v(\sigma, y) v_y(\sigma,y) \varphi(\sigma, y) \,{\rm d}y\,{\rm d}\sigma+g(t)-g_0+\int_{0}^{t}\int_{-1}^{1}v_y(\sigma, y)\varphi_y(\sigma, y)\notag\\
		=&\int_{0}^{t}\mathcal{K}(h_1-h(\sigma))\,{\rm d}\sigma+\int_{0}^{t}u(\sigma)\,{\rm d}\sigma,
	\end{align}

	Moreover, using the defintion of $\varphi$ in \eqref{testing_function} and the fact that $v(\sigma,h(\sigma))=\dot h(\sigma)$ for almost every $t\in [0, \tau),\;\sigma\in [0,\tau)$, we have
	\begin{align}\label{caculate_v_xx}
		&\int_{-1}^{1}v_y(\sigma,y)\varphi_y(\sigma,y)\,{\rm d}y\notag\\
		=&\int_{-1}^{h(\sigma)}v_y(\sigma,y)\varphi_y(\sigma,y) \,{\rm d}y+\int_{h(\sigma)}^{1}v_y(\sigma,y)\varphi_y(\sigma,y)\,{\rm d}y\notag\\
		=&\frac{\dot{h}(\sigma)}{1+h(\sigma)}+\frac{\dot{h}(\sigma)}{1-h(\sigma)}\notag\\
		=&\frac{\rm d}{{\rm d} \sigma}(\ln (1+h(\sigma)))-\frac{\rm d}{{\rm d} \sigma}(\ln (1-h(\sigma))).
	\end{align}
	
	Using \eqref{defintion_of_P}, \eqref{defintion_of_A1}, \eqref{defintion_of_A2} and \eqref{caculate_v_xx}, we can rewrite \eqref{intergrate_form} as
	\begin{align}\label{equality_of_h}
		&\ln \frac{1+h(t)}{1+h_0}-\ln \frac{1-h(t)}{1-h_0}\notag\\
		=&\int_{0}^{t}\mathcal{K}(h_1-h(\sigma))\,{\rm d}\sigma+\int_{0}^{t}u(\sigma)\,{\rm d}\sigma+P(0)-P(t)+\int_{0}^{t}A_1(\sigma ){\rm d}\sigma-\int_{0}^{t} A_2(\sigma )\,{\rm d}\sigma,
	\end{align}
	where, for almost every $t\in [0,\tau),\; \sigma \in [0,t]$, we have set
	\begin{align}
		\label{defintion_of_P}&P(\sigma )=\int_{-1}^{1}\varphi(\sigma,y) v(\sigma,y)\,{\rm d}y+g(\sigma),\\
		\label{defintion_of_A1}&A_1(\sigma )=\int_{-1}^{1}v(\sigma,y)\frac{\partial \varphi}{\partial \sigma}(\sigma,y)\,{\rm d}y,\\
		\label{defintion_of_A2}&A_2(\sigma )=\int_{-1}^{1} v(\sigma,y)v_y(\sigma,y)\varphi(\sigma,y)\,{\rm d}y.
	\end{align}
	Next, we use the energy estimate \eqref{energy_equality1} to estimate the four last terms in the right hand side of \eqref{equality_of_h}.
	
	Firstly, we remark that for $P$ defined in \eqref{defintion_of_P} we can use \eqref{energy_equality1}, the Cauchy-Schwarz inequality
	and the fact that $|\varphi|$ is uniformly bounded by $1$  to get that
	\begin{align}\label{estimate_P}
		\vert P (\sigma)\vert^2 \leq 4\left(\Vert v_0\Vert_{L^2(-1,1)}^2+g^2_0+\mathcal{K}\left(h_1-h_0\right)^2+\Vert u\Vert^2_{L^2(0,\infty)}\right) \quad (t\in [0,\tau),\  a.e. \sigma \in [0,t]).
	\end{align}
	
	To estimate $A_1$, we note that  from \eqref{defintion_of_A1} and the definition of $\varphi$ in \eqref{testing_function}, for $t\in [0,\tau))$ and almost every $\sigma\in [0,t]$ we have:
	\begin{equation}\label{with_name}
		A_1(\sigma )=-\frac{g(\sigma)}{(1+h(\sigma))^2}\, \int_{-1}^{h(\sigma)}v(\sigma,y)(1+y)\, {\rm d}y
		+\frac{g(\sigma)}{(1-h(\sigma))^2}\, \int_{h(\sigma)}^{1}
		v(\sigma,y)(1-y)\, {\rm d}y.
	\end{equation}
	To estimate the first term the right hand side of the above formula, we firstly integrate by parts  to get
	\begin{align}
		& \frac{g(\sigma)}{(1+h(\sigma))^2}\int_{-1}^{h(\sigma)}v(\sigma,y)(1+y)\,{\rm d}y \notag\\
		= &\frac{g(\sigma)}{2(1+h(\sigma))^2}\int_{-1}^{h(\sigma)}v(\sigma,y) \frac{\rm d}{{\rm d}y} \left[ (1+y)^2\right]
		\, {\rm d}y\notag\\
		= &
		- \frac{g(\sigma)}{2(1+h(\sigma))^2}
		\int_{-1}^{h(\sigma)}v_y(\sigma,y) (1+y)^2\,{\rm d}y+
		\frac12 g^2(\sigma)
		\quad \left(t\in [0,\tau),\; \sigma \in [0,t] \quad {\rm a.e.}\right),
	\end{align}
	Moreover, by combining the Cauchy-Schwarz inequality, a trace theorem, the facts $\vert 1+h(\sigma)\vert \leqslant 2$ and $v(t,h(t))=g(t)$, we have
	\begin{align*}
		&\left\vert
		\frac{g(\sigma)}{2(1+h(\sigma))^2}
		\int_{-1}^{h(\sigma)} v_y(\sigma, y) (1+y)^2 \, \mathrm{d}y - \frac{1}{2}g^2(\sigma)
		\right\vert
		\\
		\leq {}&\frac{\vert g(\sigma) \vert }{2(1+h(\sigma))^2}
		\left(
		\int_{-1}^{h(\sigma)} v_y^2(\sigma, y) \, \mathrm{d}y
		\right)^{\frac{1}{2}}
		\left(
		\int_{-1}^{h(\sigma)} (1+y)^4 \, \mathrm{d}y
		\right)^{\frac{1}{2}}
		+\frac{1}{2}g^2(\sigma)\notag\\
		\leq{} &\frac{1}{\sqrt{5}}\vert g(\sigma)\vert (1+h(\sigma))^{\frac{1}{2}}
		\left( \int_{-1}^{h(\sigma)} v_y^2(\sigma, y) \, \mathrm{d}y \right)^{\frac{1}{2}} + \frac{1}{2}g^2(\sigma) \\
		\leq{} &g(\sigma) \left( \int_{-1}^{h(\sigma)} v_y^2(\sigma, y) \, \mathrm{d}y \right)^{\frac{1}{2}} + \frac{1}{2}g^2(\sigma) \quad  \\
		\leq{} &2g^2(\sigma) + 2 \int_{-1}^{h(\sigma)} v_y^2(\sigma, y) \, \mathrm{d}y  \\
		\leq{} &6\int_{-1}^{h(\sigma)} v_y^2(\sigma, y) \, \mathrm{d}y  \qquad\qquad\qquad(t\in [0,\tau),\ \sigma\in [0,t]\; {\rm a.e.}).
	\end{align*}
	Combining the above formula and \eqref{estimateA1} it follows that
	\begin{align}\label{estimateA1}
		&\left|\frac{g(\sigma)}{(1+h(\sigma))^2}\int_{-1}^{h(\sigma)}v(\sigma,y)(1+y)\,{\rm d}y\right|\notag\\
		\leqslant& 6\int_{-1}^{h(\sigma)} v_y^2(\sigma, y) \, \mathrm{d}y  \qquad\qquad(t\in [0,\tau),\ \sigma\in [0,t]\; {\rm a.e.}).
	\end{align}
	In a completely similar manner we can estimate the second term in the right hand side of left term of \eqref{with_name}  to get
	\begin{equation*}\label{A1_second}
		\left|\frac{g(\sigma)}{(1-h(\sigma))^2}\int_{h(\sigma)}^1 v(\sigma,y)(1-y)\,{\rm d}y\right|
		\leqslant 6\int_{h(\sigma)}^1 v_y^2(\sigma, y) \, \mathrm{d}y  \qquad\qquad(t\in [0,\tau),\ \sigma\in [0,t]\; {\rm a.e.}).
	\end{equation*}
	Putting together the last two inequalities and using \eqref{energy_equality1}, it follows that
	\begin{align}\label{estimate_of_A_1}
		&\left\vert \int_{0}^{t}A_1(\sigma ){\rm d}\sigma\right\vert
		\leq\int_{0}^{t}\vert A_1(\sigma )\vert {\rm d}\sigma
		\leq 6\int_{0}^{t}\int_{-1}^{1} v^2_y(\sigma,y)\,{\rm d}y\, {\rm d}\sigma
	\notag\\	\leq& 6 \left(\Vert v_0\Vert_{L^2(-1,1)}^2+g^2_0+\mathcal{K}\left(h_1-h_0\right)^2+\Vert u\Vert^2_{L^2(0,\infty)}\right)\qquad \left(t\in [0,\tau)\right).
	\end{align}
	We next estimate of $A_2$ defined in \eqref{defintion_of_A2}. Since $\varphi$ defined by \eqref{testing_function} is uniformly bounded by $1$, using the Cauchy-Schwarz  and Poincar{\'e} inequalities, it follows that for every $t\in [0,\tau)$ and almost every $\sigma\in [0,t]$, we have
	\begin{align}\label{estimate_of_A2}
		&\vert A_2(\sigma)\vert
		\leq \int_{-1}^{1}\vert v(\sigma, y)v_y(\sigma,y)\phi(\sigma,y)\vert\, {\rm d}y\notag\\
		\leq& \left(\int_{-1}^{1} v^2(\sigma,y)\,{\rm d}y\right)^{\frac12}\left(\int_{-1}^{1} v_y^2(\sigma,y)\,{\rm d}y\right)^{\frac12}
		\leq4 \int_{-1}^{1}v^2_y(\sigma ,y)\,{\rm d}y.
	\end{align}
	From the above estimate and the energy estimate \eqref{energy_equality1} we obtain:
	\begin{align}\label{estimate_A2}
		&\left\vert \int_{0}^{t}A_2(\sigma ){\rm d}\sigma\right\vert
		\leq 4\int_{0}^{t}\int_{-1}^{1}v^2_y(\sigma ,y)\,{\rm d}y\,{\rm d}\sigma\notag\\
		\leq& 4\left(\Vert v_0\Vert_{L^2(-1,1)}+g^2_0+\mathcal{K}\left(h_1-h_0\right)^2+\Vert u\Vert^2_{L^2(0,\infty)}\right) \qquad \qquad (t\in[0,\tau)).
	\end{align}
	Moreover, using the fact $\vert h(\sigma)-h_1\vert\leq 2$, for every $t\in (0,\tau),\;\sigma \in [0,t]$, it follows that
	\begin{equation}\label{estimate_hu}
		\left\vert \int_{0}^{t}\mathcal{K}(h_1-h(\sigma)){\rm d}\sigma\right\vert \leq 2\mathcal{K}t\qquad \quad 	\left\vert \int_{0}^{t}u(\sigma){\rm d}\sigma\right\vert \leq t^{\frac12}\Vert u\Vert_{L^2(0,\infty)}.
	\end{equation}
	By combing  \eqref{estimate_P}, \eqref{estimate_of_A_1}, \eqref{estimate_A2}, \eqref{estimate_hu} and \eqref{equality_of_h}, we have
	\begin{align}\label{caculation_of_h}
		\ln \frac{1+h(t)}{1+h_0}-\ln \frac{1-h(t)}{1-h_0}\leq C+2\mathcal{K}t + t^{\frac12}\Vert u\Vert_{L^2(0,\infty)} \qquad \quad (t\in[0,\tau)),
	\end{align}
	where 
	\begin{align}\label{defintion_C}
	C=&10\left(\Vert v_0\Vert_{L^2(-1,1)}^2+g^2_0+\mathcal{K}(h_1-h_0)^2
	\right.\notag\\
&\left.	+ \Vert u\Vert^2_{L^2(0,\infty)}+\sqrt{\Vert v_0\Vert_{L^2(-1,1)}^2+g^2_0+\mathcal{K}(h_1-h_0)^2+ \Vert u\Vert^2_{L^2(0,\infty)}}\right).
	\end{align}
Since for every $t>0$, we have
\begin{equation*}
    \frac{1+h_0}{1-h_0}\exp\left(C+2\mathcal{K}t+ t^{\frac12}\Vert u\Vert_{L^2(0,\infty)}\right)\leq \max\left\{2,\frac{1+h_0}{1-h_0}\right\}\exp\left(C+2\mathcal{K}t+ t^{\frac12}\Vert u\Vert_{L^2(0,\infty)}\right),
\end{equation*}
	then after a simple calculation, the above inequality \eqref{caculation_of_h} can be rephrased to
	\begin{equation}\label{estimate_h_upper}
		h(t)\leq 1-\frac{2}{1+\max\left\{2,\frac{1+h_0}{1-h_0}\right\}\exp\left(C+2\mathcal{K}t+ t^{\frac12}\Vert u\Vert_{L^2(0,\infty)}\right)}\qquad\qquad (t\in[0,\tau)).
	\end{equation}
	Choosing next the test function $\begin{bmatrix}
		\psi\\
		l
	\end{bmatrix}=\begin{bmatrix}
		-\varphi\\
		1
	\end{bmatrix}$ in Defintion \ref{definition_of_finite_energy} and following step by step the procedure used to prove \eqref{estimate_h_upper}, it follows that:
	\begin{equation}\label{estimate_h_lower}
		h(t)\geq -1+\frac{2}{1+\max\left\{2,\frac{1-h_0}{1+h_0}\right\}\exp\left(C+2\mathcal{K}t+t^{\frac12}\Vert u\Vert_{L^2(0,\infty)}\right)}   \qquad\qquad (t\in[0,\tau)).
	\end{equation}

By combining  Corollary \ref{corollary_contraction} and Proposition \ref{local_existstheorem} , we finish the proof of Theorem \ref{global_exists}.
\end{proof}

\section{Proof of Theorem \ref{unforimboundtheorem}}\label{sectionuniformbound}

This section is devoted to the proof of one of our main results.

\begin{proof}[Proof of Theorem \ref{unforimboundtheorem}]
Let 
$\begin{bmatrix}
    v\\
    g\\
    h
\end{bmatrix}$
be the finite energy solution defined by \eqref{definition_of_finite_energy} to system \eqref{conversition_of_mom_of_solid_without_control}–\eqref{inlitial_date_of_solid_without_control}. Then the energy inequality \eqref{energy_equality1} holds for every $t \geq 0$.

From this inequality we know that
$$
\dot{h}=g\in L^\infty(0,\infty),
$$
together with the estimate \eqref{estimateh} and the Sobolev embedding theorem, we obtain
\[
h \in W^{1,\infty}(0,\infty) \hookrightarrow C[0,\infty),
\]
which implies that $h$ is almost continuous on $[0,\infty)$.

We now prove Theorem \ref{unforimboundtheorem}. Since global existence and uniqueness have already been established, it remains to prove the uniform bound \eqref{estimateh1} on the particle’s position $h(t)$.

By the continuity of $h$ and the local sign-preserving property of continuous functions, it suffices to verify that $h(t)$ stays uniformly away from the value 1. This can be done by considering three possible scenarios:

\begin{itemize}
    \item[(1)] For all $t > 0$,
    \[
    h(t) \leq  h_1.
    \]

    \item[(2)]  If there exists $T^* > 0$ such that
    \[
    h(T^*) = h_1 ,
    \]
    and for all $t > T^*$,
    \[
    h(t) >  h_1.
    \]
    Then there exists a constant $\alpha > 0$ such that
    \[
    h(t) \leq 1 - \alpha, \quad \forall t > 0.
    \]

    \item[(3)] If there exists a sequence $\{T_i^*\}_{i=1}^N$ ($N$ could be infinite) such that
    \[
    h(T_i^*) = h_1,
    \]
    and for each $t \in (T_i^*, T_{i+1}^*)$,
    \[
    h(t) > h_1.
    \]
    Then again, there exists a uniform constant $\alpha > 0$, independent of $i$, such that
    \[
    h(t) \leq 1 - \alpha, \quad \forall t \in (T_i^*, T_{i+1}^*).
    \]
\end{itemize}

\textbf{Case 1.} The bound is already satisfied, so there is nothing to prove.

\medskip

\textbf{Case 2.} From equation \eqref{equality_of_h}, we have:
\[
\begin{aligned}
\ln \frac{1 + h(t)}{1 + h_0} - \ln \frac{1 - h(t)}{1 - h_0}
&= \int_0^t \mathcal{K}(h_1 - h(\sigma)) \, d\sigma + \int_0^t u(\sigma) \, d\sigma \\
&\quad + P(0) - P(t) + \int_0^t A_1(\sigma) \, d\sigma - \int_0^t A_2(\sigma) \, d\sigma.
\end{aligned}
\]

Since $h(t) > h_1$ for $t > T^*$, we have $\mathcal{K}(h_1 - h(t)) \leq  0$ on $(T^*, \infty)$.

Then for $t > T^*$, we rewrite the above equation as:
\[
\begin{aligned}
\ln \frac{1 + h(t)}{1 + h_0} - \ln \frac{1 - h(t)}{1 - h_0}
=& \ln \frac{1 + h_1}{1 + h_0} - \ln \frac{1 - h_1}{1 - h_0} \\
&+ \int_{T^*}^t \mathcal{K}(h_1 - h(\sigma)) \, d\sigma + \int_{T^*}^t u(\sigma) \, d\sigma \\
&+ P(T^*) - P(t) + \int_{T^*}^t A_1(\sigma) \, d\sigma - \int_{T^*}^t A_2(\sigma) \, d\sigma.
\end{aligned}
\]

Since $\mathcal{K}(h_1 - h(\sigma)) \leq 0$, from the above equality, we get:
\[
\begin{aligned}
\ln \frac{1 + h(t)}{1 + h_0} - \ln \frac{1 - h(t)}{1 - h_0}
\leq& \ln \frac{1 + h_1}{1 + h_0} - \ln \frac{1 - h_1}{1 - h_0} \\
&+ \int_{T^*}^t u(\sigma) \, d\sigma + P(T^*) - P(t) \\
&+ \int_{T^*}^t A_1(\sigma) \, d\sigma - \int_{T^*}^t A_2(\sigma) \, d\sigma.
\end{aligned}
\]

By assumption $u\in L^1(0, \infty)$ and previous estimates \eqref{estimate_P}, \eqref{estimate_of_A_1}, \eqref{estimate_A2}, all terms on the right-hand side are bounded by $\|v_0\|_{L^2(-1,1)}$, $g_0$, $h_0$, $h_1$, $\mathcal{K}$, $\|u\|_{L^2(0,\infty)}$, and $\|u\|_{L^1(0,\infty)}$. Therefore,
\[
\ln \frac{1 + h(t)}{1 + h_0} - \ln \frac{1 - h(t)}{1 - h_0} \leq C,
\]
for some constant $C > 0$. This implies
\[
h(t) \leq 1 - \alpha,
\]
with
\begin{equation}\label{definitionofalpha}
    \alpha=\frac{2}{1+C}
\end{equation}
where
\begin{align}\label{defintion_alpha_C}
	C=&\frac{1+h_0}{1-h_0}\exp\left(\ln \frac{1 + h_1}{1 + h_0} - \ln \frac{1 - h_1}{1 - h_0}+10\left(\Vert v_0\Vert_{L^2(-1,1)}^2+g^2_0+\mathcal{K}(h_1-h_0)^2
	\right.\right.\notag\\
&\left.\left.	+ \Vert u\Vert^2_{L^2(0,\infty)}+\sqrt{\Vert v_0\Vert_{L^2(-1,1)}^2+g^2_0+\mathcal{K}(h_1-h_0)^2+ \Vert u\Vert^2_{L^2(0,\infty)}}\right)+\Vert u\Vert_{L^1(0,\infty)}\right).
	\end{align}
\medskip

\textbf{Case 3.} The same procedures as in Case 2 can be applied on each interval $(T_i^*, T_{i+1}^*)$, since $h(T_i^*) = \beta$ and $h(t) > \beta$ inside the interval. Therefore, the same bound holds on each sub-interval, with a uniform constant $\alpha > 0$ which has been defined in \eqref{definitionofalpha} and \eqref{defintion_alpha_C}.

\medskip

\noindent Combining the three cases, we conclude that there exists a constant $\alpha > 0$ such that
\begin{equation}\label{upperbound1}
    h(t) \leq 1 - \alpha, \quad \forall t > 0.
\end{equation}

Moreover, since $h$ is almost continuous, we can follow the same steps as in the proof of \eqref{upperbound1} to show that there exists a constant $\alpha$ such that the position trajectory $h$ of the system \eqref{conversition_of_mom_of_solid_without_control}--\eqref{inlitial_date_of_solid_without_control} satisfies the desired estimate.
\begin{equation}\label{lowerbound1}
    1 - \alpha \leq h(t), \quad \forall t > 0.
\end{equation}

Finally, combining \eqref{upperbound1}, \eqref{lowerbound1} and Theorem \ref{global_exists}, we complete the proof of Theorem \ref{unforimboundtheorem}.
\end{proof}

\section{Proof of Theorem \ref{theoreminputtostate}}\label{section_iss}

We first give the proof of inequality \eqref{input_states_stability1}.
\begin{proof}[Proof of Theorem \ref{theoreminputtostate}]
 Let $\varepsilon$ be such that
\begin{equation}\label{pripa_eps}
0  \leq  \varepsilon\leq \min\left\{\frac18, \frac{\mathcal{K}}{8}\right\}\qquad \qquad (\mathcal{K}>0),
\end{equation}
and  we define
	\begin{align}\label{Lyapunov_functions_of_interaction_system}
		V_\varepsilon(t)
		=&	\int_{-1}^{1}v^2(t,y)\,{\rm d}y+g^2(t)+\mathcal{K}\left(h(t)-h_1\right)^2\notag\\
		&-\varepsilon\left(h_1-h(t)\right) P(t) \qquad\qquad\qquad \left(t\geqslant 0\right),
	\end{align}
	where  $\begin{bmatrix}
		v\\
		g\\
		h
	\end{bmatrix}$ is the global finite energy solution of system \eqref{conversition_of_mom_of_solid_without_control}-\eqref{inlitial_date_of_solid_without_control}, which has been defined in Definition \ref{definition_of_finite_energy} and $P$ has been defined in \eqref{defintion_of_P}.

	Meanwhile, for every $t\geq 0$, we also define
	\begin{equation}\label{definition_of_energy}
		E(t)=\int_{-1}^{1}v^2(t,y)\,{\rm d}y+g^2(t)+\mathcal{K}\left(h(t)-h_1\right)^2\qquad (\mathcal{K}> 0).
	\end{equation}
	By combining  the Cauchy-Schwarz inequality and the fact $\vert\varphi\vert$ is uniform bounded by $1$, it is straightforward to verify that,
	\begin{equation}\label{norm_inequality2}
		\frac14 E(t)\leq V_{\varepsilon}(t)\leq 2E(t),
		\qquad
		\frac14 V_{\varepsilon}(t) \leq E(t)\leq 2 V_{\varepsilon}(t) \qquad\qquad (t\geq 0).
	\end{equation}
We next remark that $V_{\varepsilon}(t)$ is differentiable with respect to $t$.
Indeed, from the energy estimate given by \eqref{energy_equality} and Theorem \ref{global_exists}, we deduce that the mapping
\begin{equation}
	t \mapsto \frac{1}{2}\left(\int_{-1}^{1} v^2(t,y) \, \mathrm{d}y + g^2(t) + \mathcal{K}(h(t) - h_1)^2\right)
\end{equation}
is absolutely continuous and differentiable for almost every $t \in [0, \infty)$. Consequently, by taking the derivative of \eqref{definition_of_energy} with respect to $t$, we obtain
\begin{align}\label{energyestimatesection5}
	\frac{{\rm d}}{\rm d t}E(t)= -2\int_{-1}^1v_y^2(t,y) {\rm d } y+g(t)u(t) \qquad \qquad ( t\geq 0\;{\rm a.e.}).
\end{align}
On the other hand, from \eqref{equality_of_h} it follows that $P$ is absolutely continuous on every bounded interval contained in $[0,\infty)$,
thus differentiable almost everywhere on $[0,\infty)$. Moreover, from \eqref{defintion_of_P} it also follows that for almost every $t\geqslant 0$
we have
\begin{align}\label{derivative_expression}
	\frac{\mathrm{d}P}{\mathrm{d}t}(t)
	=& A_1(t) - A_2(t) - \int_{-1}^{1}v_y(t,y)\varphi_y(t,y) \, \mathrm{d}y + \mathcal{K}(h_1 - h(t))\notag\\
	=& A_1(t) - A_2(t) + \mathcal{K}(h_1 - h(t)) - \left(\frac{g(t)}{1+h(t)} + \frac{g(t)}{1-h(t)}\right)+u(t),
\end{align}
where $A_1$ and $A_2$ have been defined in \eqref{defintion_of_A1} and \eqref{defintion_of_A2}, respectively.

 From \eqref{energyestimatesection5} and \eqref{derivative_expression} we get that $V_\varepsilon(t)$ defined by \eqref{Lyapunov_functions_of_interaction_system} is absolutely continuous, thus differentiable for almost every $t \in [0, \infty)$, with
 \begin{align}\label{deofV}
 	\frac{\mathrm{d}}{\mathrm{d}t}V_\varepsilon(t)
 	=& -2\int_{-1}^{1} v_y^2(t,y) \, \mathrm{d}y
 	+g(t)u(t) + \varepsilon g(t) \left( g(t) + \int_{0}^{1} \varphi(t,y) v(t,y) \, \mathrm{d}y \right)  
 	 - \varepsilon \left( h_1 - h(t) \right) \notag\\
     &\left( A_1(t) - A_2(t) + \mathcal{K}(h_1 - h(t)) - \left( \frac{g(t)}{1 + h(t)} + \frac{g(t)}{1 - h(t)} \right) +u(t)\right).
 \end{align}
 By using Cauchy-Schwarz inequality, for every $\mathcal{K}>0$ we have 
 \begin{equation}
     g(t)u(t)\leq \frac12 g^2(t)+ \frac12 u^2(t),\qquad  (h_1-h(t))u(t)\leq \frac{\mathcal{K}}{4}(h_1-h(t))^2+\frac{1}{\mathcal{K}}u^2(t)\qquad (t\geq 0).
 \end{equation}
Moreover, we know from \eqref{estimateh} that there exists a constant $\alpha \in (0,1]$, such that
\begin{equation*}
	\min\{ 1-h(t), \;1+h(t)\}\geq \alpha\qquad (t\geq 0).
\end{equation*}
The above estimate, combined with an elementary  inequality, shows that for every $t\geq 0$ we have
\begin{equation}\label{estimateoffrac}
	\left(h_1-h(t)\right)\left( \frac{g(t)}{1 + h(t)} + \frac{g(t)}{1 - h(t)} \right)\leq \frac{g^2(t)}{\mathcal{K}\alpha^2}+\frac{\mathcal{K}}{4}(h_1-h(t))^2\qquad (\mathcal{K}>0).
\end{equation}
From the above inequality it follows that for every $\varepsilon$ satisfying \eqref{pripa_eps}
and almost every $t\geqslant 0$ and $\mathcal{K}>0$ we have
\begin{align*}
\frac{\mathrm{d}}{\mathrm{d}t}V_{\varepsilon}(t)
	\leq& \left(-\int_{-1}^{1}v_y^2(t,y)\,\mathrm{d}y - \varepsilon(h_1 - h(t))A_1(t)
+ \varepsilon(h_1 - h(t))A_2(t)\right)\\
&+\frac12 g^2(t)+ \left(5 + \frac{1}{\mathcal{K}\alpha^2}\right)\varepsilon g^2(t)
	- \frac{3\mathcal{K}\varepsilon}{4}(h_1 - h(t))^2+\left(\frac12+\frac{\varepsilon}{\mathcal{K}}\right)u^2(t).
\end{align*}

		Using the inequality above and the fact that $|h(t) - h_1| \leq 2$, along with the estimates of $A_1$ and $A_2$ in \eqref{estimateA1}, \eqref{estimate_A2}, respectively, we have
		\begin{align}\label{are_acuma}
		\frac{\mathrm{d}}{\mathrm{d}t}V_{\varepsilon}(t)
            \leq& \left(-2 + 20\varepsilon\right)\int_{-1}^{1}v_y^2(t,y)\,\mathrm{d}y +\frac12 g^2(t)+ \left(5 + \frac{1}{\mathcal{K}\alpha^2}\right)\varepsilon g^2(t) - \frac{3\mathcal{K}\varepsilon}{4}(h_1 - h(t))^2\notag\\
            &+\left(\frac12+\frac{\varepsilon}{\mathcal{K}}\right)u^2(t)
		\end{align}
Using the estimates
	\begin{align}\label{here_label}
		&g(t)=v(t,h(t))=\int_{-1}^{h(t)}v_y(t,y)\,{\rm d}y\leq \left(\int_{-1}^{1}v_y^2(t,y)\,{\rm d}y\right)^\frac{1}{2}\left(\int_{-1}^{1}1\,{\rm d}y\right)^\frac{1}{2}\notag\\
		\leq& \sqrt{2}\left(\int_{-1}^{1}v_y^2(t,y)\,{\rm d}y\right)^\frac{1}{2},
	\end{align}	
inequality \eqref{are_acuma} implies that for almost every $t\geqslant 0$ we have
\begin{align}
&\frac{\mathrm{d}}{\mathrm{d}t}V_{\varepsilon}(t) 
\leq \left(-1 + 20\varepsilon\right)\int_{-1}^{1}v_y^2(t,y)\,\mathrm{d}y + \left(5 + \frac{1}{\mathcal{K}\alpha^2}\right)\varepsilon g^2(t)\notag\\
&- \frac{3\mathcal{K}\varepsilon}{4}(h_1 - h(t))^2+\left(\frac12+\frac{\varepsilon}{\mathcal{K}}\right)u^2(t)\notag\\
=& \left(-1+ 20\varepsilon\right)\int_{-1}^{1}v_y^2(t,y)\,\mathrm{d}y + \left(7 + \frac{1}{\mathcal{K}\alpha^2}\right)\varepsilon g^2(t)-2\varepsilon g^2(t)\notag\\
&- \frac{3\mathcal{K}\varepsilon}{4}(h_1 - h(t))^2+\left(\frac12+\frac{\varepsilon}{\mathcal{K}}\right)u^2(t)\notag\\
\leq &\left(-1+20\varepsilon+2\left(7 + \frac{1}{\mathcal{K}\alpha^2}\right)\varepsilon\right)\int_{-1}^{1}v_y^2(t,y)\,{\rm d}y-2\varepsilon g^2(t)\notag\\
&- \frac{3\mathcal{K}\varepsilon}{4}(h_1 - h(t))^2+\left(\frac12+\frac{\varepsilon}{\mathcal{K}}\right)u^2(t).
\end{align}
		We choose $\varepsilon = \frac{1}{16\left(34 + \frac{2}{\mathcal{K}\alpha^2}\right)}$, which satisfies  \eqref{pripa_eps}. Combining this with the Poincar\'e inequality, we obtain
		\begin{align}
			\frac{\mathrm{d}}{\mathrm{d}t}V_{\varepsilon}(t) \leq& -\frac{1}{4}\left[1 - \left(34 + \frac{2}{\mathcal{K}\alpha^2}\right)\varepsilon\right]\int_{-1}^{1}v^2(t,y)\,\mathrm{d}y - 2\varepsilon g^2(t) \notag\\
            &- \frac{3\mathcal{K}\varepsilon}{4}(h_1 - h(t))^2+\left(\frac12+\frac{\varepsilon}{\mathcal{K}}\right)u^2(t),
		\end{align}
	since $\alpha\leq 1$, then we have
    $$\frac{\varepsilon}{\mathcal{K}}= \frac{1}{544\mathcal{K}+\frac{2}{\alpha^2}}\leq 1$$
		Denote
		\begin{equation}\label{definition_of_lambda}
			\eta = \frac{1}{4} \min\left\{ \frac{1}{34 + \frac{2}{\mathcal{K}\alpha^2}}, \; \frac{3\mathcal{K}\varepsilon}{4} \right\}.
		\end{equation}
		
		Using the above inequality and \eqref{norm_inequality2}, for almost every $t \geq 0$, we have
		\begin{equation}
			\frac{\mathrm{d}}{\mathrm{d}t}V_{\varepsilon}(t) \leq -4\eta E(t) +\frac32 u^2(t)\leq -\eta V_{\varepsilon}(t)+\frac32 u^2(t),
		\end{equation}
		where $E(t)$ is defined by \eqref{definition_of_energy}.
		
		By Gronwall's inequality, for $\mathcal{K} > 0$, we obtain
		\begin{equation}\label{use_end}
			V_{\varepsilon}(t) \leq \exp(-\eta t)V_{\varepsilon}(0)+\frac32 \Vert u\Vert^2_{L^2(0,\infty)}.
		\end{equation}
		
		Using \eqref{norm_inequality2} again, we conclude that if $\mathcal{K}>0$ then
		\begin{equation}\label{energy_decay_K_positive}
			E(t) \leq 16\exp(-\eta t)E(0) +\frac32\Vert u\Vert^2_{L^2(0,\infty)}\qquad\qquad (t \geqslant 0).
		\end{equation}

        Here, we finish the proof of inequality \eqref{input_states_stability1}.

    Next, for $\mathcal{K} = 0$, using \eqref{energyestimatesection5}, \eqref{here_label}, and the Cauchy--Schwarz inequality, we obtain
\begin{align}
    \frac{{\rm d}}{{\rm d}t}\left(\|v(t)\|^2_{L^2(-1,1)} + g^2(t)\right) 
    &\leq -2\int_{-1}^{1} |v_y(t,y)|^2\,{\rm d}y + \frac{1}{2}g^2(t) + \frac{1}{2}u^2(t) \notag \\
    &\leq -2\int_{-1}^{1} |v_y(t,y)|^2\,{\rm d}y + \frac{3}{4}g^2(t) - \frac{1}{4}g^2(t) + \frac{1}{2}u^2(t) \notag \\
    &\leq -\frac{1}{2}\int_{-1}^{1} |v_y(t,y)|^2\,{\rm d}y - \frac{1}{4}g^2(t) + \frac{1}{2}u^2(t).
\end{align}

Applying the Poincar\'e  inequality on $(-1,1)$, we have
\begin{equation}
    \int_{-1}^{1} |v(t,y)|^2\,{\rm d}y \leq C_P \int_{-1}^{1} |v_y(t,y)|^2\,{\rm d}y,
\end{equation}
where $C_P = \left(\frac{1}{\pi/2}\right)^2 = \frac{4}{\pi^2}$, see for Payne and Weinberger \cite{payne1960optimal}. For simplicity, we absorb the constant into the inequality and write
\begin{equation}
    \frac{{\rm d}}{{\rm d}t}\left(\|v(t,\cdot)\|^2_{L^2(-1,1)} + g^2(t)\right) 
    \leq -\frac{1}{4} \|v(t,\cdot)\|^2_{L^2(-1,1)} - \frac{1}{4}g^2(t) + \frac{1}{2}u^2(t).
\end{equation}

Now, applying Gronwall’s inequality, we deduce that
\begin{equation*}
    \|v(t,\cdot)\|^2_{L^2(-1,1)} + g^2(t)
    \leq \exp \left({-\frac{1}{4}t}\right)\left( \|v_0\|^2_{L^2(-1,1)} + g_0^2 \right) + \frac{1}{2} \int_0^t u^2(\sigma)\,{\rm d}\sigma.
\end{equation*}

Combining the estimate above with \eqref{energy_decay_K_positive}, we complete the proof of Theorem~\ref{theoreminputtostate}.
 \end{proof}  
     Finally, we introduce a standard local input-to-state stability for system \eqref{conversition_of_mom_of_solid_without_control}-\eqref{inlitial_date_of_solid_without_control}.
\begin{proposition}   
[local eISS]
Assume that $v_0 \in L^2(-1,1)$, $g_0 \in \mathbb{R}$, $h_0, h_1 \in (-1,1)$, and $u \in L^2(0,\infty)$ satisfy:
\begin{align}
\label{condition1}
&0 < |h_0 - h_1| < \frac{1}{2\sqrt{2}} \min\{1 - h_1, 1 + h_1\}, \\
\label{condition2}
&\mathcal{K} > \frac{\|v_0\|^2_{L^2(-1,1)} + g_0^2 + \|u\|^2_{L^2(0,\infty)}}{|h_0 - h_1|^2}.
\end{align}
Then system \eqref{conversition_of_mom_of_solid_without_control}--\eqref{inlitial_date_of_solid_without_control} is exponentially input-to-state stable. 
That is, there exist constants 

\noindent$C_1(h_1) > 0$ such that for all $t \geq 0$, we have
\begin{align}
\label{local input_states_stability3}
&\|v(t)\|^2_{L^2(-1,1)} + g^2(t) + \mathcal{K}(h(t) - h_1)^2  \notag\\
\leq& 16 e^{-C_1 t} \left( \|v_0\|^2_{L^2(-1,1)} + g_0^2 + \mathcal{K}(h_0 - h_1)^2 \right) + \frac32 \|u\|^2_{L^2(0,\infty)}.
\end{align}
\end{proposition}

\begin{proof}
From the energy estimate and assumption \eqref{condition2}, we have:
\[
\int_{-1}^{1} v^2(t,y)\, \mathrm{d}y + g^2(t) + \mathcal{K}(h(t) - h_1)^2 \leq 2\mathcal{K}|h_0 - h_1|^2.
\]
Then,
\[
|h(t) - h_1| \leq \sqrt{2} |h_0 - h_1| < \frac{1}{2} \min(1 - h_1, 1 + h_1),
\]
which implies
\[
\min(1 - h(t), 1 + h(t)) \geq \frac{1}{2} \min(1 - h_1, 1 + h_1).
\]
Taking $\alpha = \frac{1}{2} \min(1 - h_1, 1 + h_1)$, we can apply the same analysis from \eqref{deofV}-\eqref{energy_decay_K_positive} to get \eqref{local input_states_stability3}. Then combined the estimate \eqref{local input_states_stability3} with the definition of ISS in Section \ref{background_of_input}, we get the local ISS property of system \eqref{conversition_of_mom_of_solid_without_control}-\eqref{inlitial_date_of_solid_without_control}.
\end{proof}

{\bf Acknowledgments.}
This work is Funded by the European Union (Horizon Europe MSCA project Modconflex, grant number 101073558).

%%%%%%%%%%%%%%%%%
\bibliographystyle{siam}
\bibliography{references}

\begin{thebibliography}{10}

\bibitem{cindea2015particle}
{\sc N.~C{\^\i}ndea, S.~Micu, I.~Roven{\c{t}}a, and M.~Tucsnak}, {\em Particle
  supported control of a fluid--particle system}, Journal de Math{\'e}matiques
  Pures et Appliqu{\'e}es, 104 (2015), pp.~311--353.

\bibitem{doubova2005some}
{\sc A.~Doubova and E.~Fern{\'a}ndez-Cara}, {\em Some control results for
  simplified one-dimensional models of fluid-solid interaction}, Mathematical
  Models and Methods in Applied Sciences, 15 (2005), pp.~783--824.

\bibitem{Helsa2004PhD75H}
{\sc T.~I. {Hesla}}, {\em {Collisions of smooth bodies in viscous fluids: A
  mathematical investigation}}, PhD thesis, University of Minnesota, Twin
  Cities, Jan. 2004.

\bibitem{tucsnak2023input}
{\sc R.~Hosfeld, B.~Jacob, F.~L. Schwenninger, and M.~Tucsnak}, {\em
  Input-to-state stability for bilinear feedback systems}, SIAM J. Control
  Optim., 62 (2024), pp.~1369--1389.

\bibitem{jacob2022input}
{\sc B.~Jacob, A.~Mironchenko, and F.~Schwenninger}, {\em Input-to-state
  stability for infinite-dimensional systems}, Mathematics of Control, Signals,
  and Systems, 34 (2022), pp.~215--216.

\bibitem{jacob2016input}
{\sc B.~Jacob, R.~Nabiullin, J.~Partington, and F.~Schwenninger}, {\em On
  input-to-state-stability and integral input-to-state-stability for parabolic
  boundary control systems}, in 2016 IEEE 55th Conference on Decision and
  Control (CDC), IEEE, 2016, pp.~2265--2269.

\bibitem{jacob2018infinite}
{\sc B.~Jacob, R.~Nabiullin, J.~R. Partington, and F.~L. Schwenninger}, {\em
  Infinite-dimensional input-to-state stability and orlicz spaces}, SIAM
  Journal on Control and Optimization, 56 (2018), pp.~868--889.

\bibitem{karafyllis2016input}
{\sc I.~Karafyllis and M.~Krstic}, {\em Input-to-state stability with respect
  to boundary disturbances for the 1-d heat equation}, in 2016 IEEE 55th
  Conference on Decision and Control (CDC), IEEE, 2016, pp.~2247--2252.

\bibitem{karafyllis2016iss}
\leavevmode\vrule height 2pt depth -1.6pt width 23pt, {\em Iss with respect to
  boundary disturbances for 1-d parabolic pdes}, IEEE Transactions on Automatic
  Control, 61 (2016), pp.~3712--3724.

\bibitem{karafyllis2017iss}
\leavevmode\vrule height 2pt depth -1.6pt width 23pt, {\em Iss in different
  norms for 1-d parabolic pdes with boundary disturbances}, SIAM Journal on
  Control and Optimization, 55 (2017), pp.~1716--1751.

\bibitem{karafyllis2021iss}
\leavevmode\vrule height 2pt depth -1.6pt width 23pt, {\em I{SS} estimates in
  the spatial sup-norm for nonlinear 1-{D} parabolic {PDE}s}, ESAIM Control
  Optim. Calc. Var., 27 (2021), pp.~Paper No. 57, 23.

\bibitem{liu2013single}
{\sc Y.~Liu, T.~Takahashi, and M.~Tucsnak}, {\em Single input controllability
  of a simplified fluid-structure interaction model}, ESAIM: Control,
  Optimisation and Calculus of Variations, 19 (2013), pp.~20--42.

\bibitem{mironchenko2019monotonicity}
{\sc A.~Mironchenko, I.~Karafyllis, and M.~Krstic}, {\em Monotonicity methods
  for input-to-state stability of nonlinear parabolic pdes with boundary
  disturbances}, SIAM Journal on Control and Optimization, 57 (2019),
  pp.~510--532.

\bibitem{inputoverview}
{\sc A.~Mironchenko and C.~Prieur}, {\em Input-to-state stability of
  infinite-dimensional systems: recent results and open questions}, SIAM Rev.,
  62 (2020), pp.~529--614.

\bibitem{payne1960optimal}
{\sc L.~E. Payne and H.~F. Weinberger}, {\em An optimal poincar{\'e} inequality
  for convex domains}, Archive for Rational Mechanics and Analysis, 5 (1960),
  pp.~286--292.

\bibitem{sontag1989smooth}
{\sc E.~D. Sontag et~al.}, {\em Smooth stabilization implies coprime
  factorization}, IEEE transactions on automatic control, 34 (1989),
  pp.~435--443.

\bibitem{sontag1995characterizations}
{\sc E.~D. Sontag and Y.~Wang}, {\em On characterizations of the input-to-state
  stability property}, Systems \& Control Letters, 24 (1995), pp.~351--359.

\bibitem{tucsnak2014well}
{\sc M.~Tucsnak and G.~Weiss}, {\em Well-posed systems---the {LTI} case and
  beyond}, Automatica J. IFAC, 50 (2014), pp.~1757--1779.

\bibitem{vazquez2003large}
{\sc J.~L. V{\'a}zquez and E.~Zuazua}, {\em Large time behavior for a
  simplified 1d model of fluid--solid interaction},  (2003).

\bibitem{zuazua2006}
{\sc J.~L. V{\'a}zquez and E.~Zuazua}, {\em Lack of collision in a simplified
  1d model for fluid?solid interaction}, Mathematical Models and Methods in
  Applied Sciences, 16 (2006), pp.~637--678.

\bibitem{xu5114347global}
{\sc Z.~Xu and M.~Tucsnak}, {\em Global exponential stabilization for a
  simplified fluid-particle interaction system}, Available at SSRN 5114347.

\end{thebibliography}

\end{document}